\documentclass{birkjour}
\usepackage[T1]{fontenc}
\usepackage{cite}
\usepackage{amscd,amssymb,amsmath,color}
\usepackage[all]{xy}

\def\C{\mathbb C}
\def\N{\mathbb N}

\def\T{\mathbb T}
\def\Z{\mathbb Z}

\def\H{\mathcal H}
\def\K{\mathcal K}
\def\O{\mathcal O}

\def\h{\mathfrak h}

\def\aut{\operatorname{Aut}}

\def\id{\operatorname{id}}

\def\pol{\mathcal{O}}
\def\lra{\longrightarrow}
\def\ra{\rightarrow}

\def\tens{\mathop{\otimes}} 
\def\ot{\mathop{\otimes}}

\newtheorem{theo}{Theorem}
\newtheorem{coro}[theo]{Corollary}
\newtheorem{lemm}[theo]{Lemma}

\newcounter{rlist}
 \newenvironment{rlist}{\begin{list}{\rm{(\roman{rlist})}}{
 \usecounter{rlist}\leftmargin2.5em\labelwidth2em\labelsep0.5em
 \topsep0.6ex
 \parsep0.3ex plus0.2ex minus0.1ex}}{\end{list}}

\begin{document}

\title{On the quantum flag manifold $SU_q(3)/\T^2$}

 \author{Tomasz Brzezi\'nski}
 \address{Department of Mathematics, Swansea University\\
Swansea University Bay Campus\\
Fabian Way,
Swansea
SA1 8EN, U.K.  \\
\and\\ Department of Mathematics\\ University of Bia\l ystok\\ K.\ Cio\l kowskiego 1M\\ 15--245 Bia\l ystok\\ Poland} 
  \email{T.Brzezinski@swansea.ac.uk}   
\author{Wojciech Szyma\'nski}
\address{Department of Mathematics and Computer Science\\ University of Southern Denmark\\
Campusvej 55\\ 5230 Odense M, Denmark} 
\email{szymanski@imada.sdu.dk} 

\thanks{The research of the first named author was partially supported by the Polish National Science Centre grant 2016/21/B/ST1/02438. The second named author was supported by  the 
DFF-Research Project 2, `Automorphisms and invariants of operator algebras', Nr. 7014--00145B, 2017--2021.}

\subjclass{46L65, 81R60}
\keywords{Quantum flag manifold, quantum fibre bundle, quantum SU(3) group}

\begin{abstract}
The structure of the $C^*$-algebra of functions on the quantum flag manifold $SU_q(3)/\T^2$ is investigated. Building on the representation 
theory of $C(SU_q(3))$, we analyze irreducible representations and the primitive ideal space of $C(SU_q(3)/\T^2)$, with a view towards unearthing 
the ``quantum sphere bundle'' $\C P_q^1 \to SU_q(3)/\T^2 \to \C P_q^2$. 
\end{abstract}

\maketitle

\section{Introduction}

The theory of principal and associated fibre bundles lies at the heart of geometry and underpins important applications to physics. Due to combined 
effort of many researchers, see e.g.\ \cite{BrzMaj:gau,h,BrzHaj:Che}, this theory has been successfully incorporated 
into noncommutative geometry.   In the noncommutative setting, spaces are replaced by (noncommutative) algebras of functions, typically 
$C^*$-algebras or their dense $*$-subalgebras, and quantum groups (or Hopf algebras) play the role of structure groups. By contrast, 
precious little is known about noncommutative analogs of more general fibre bundles, in which the fibre does not correspond to a group. 

This short note is intended as a first step towards a case study of  noncommutative sphere bundles. More specifically, the classical flag manifold 
$SU(3)/\T^2$ has a natural structure of the sphere bundle 
$$ \C P^1 \to SU(3)/\T^2 \to \C P^2. $$
We intend to analyze the structure of the quantum analog of this flag manifold, corresponding to the $C^*$-algebra $C(SU_q(3)/\T^2)$ 
playing the role of the total space. 
Here $SU_q(3)$ denotes the Woronowicz quantum $SU(3)$ group, and $C(SU_q(3)/\T^2)$ itself is the $C^*$-algebra of fixed points for 
the action of $\T^2$ on $C(SU_q(3))$ coming from its maximal torus. 

The quantum flag manifold $SU_q(3)/\T^2$ is just one of the large family of (generalised) quantum flag manifolds, whose structure has been studied and described in full generality in \cite{StoDij:fla} and \cite{NevTus:hom}. However, in order to be able to understand $SU_q(3)/\T^2$ as the total space of a quantum sphere bundle from the analytic point of view, it is necessary to have a detailed and explicit information about the internal structure of the $C^*$-algebra $C(SU_q(3)/\T^2)$ readily accessible. 
This is the main aim of the present note.
In particular, we carefully describe the 
primitive ideal space of this $AF$-algebra, building on the explicit description of irreducible representations of $C(SU_q(3))$, calculated 
originally by K.\ Br\k{a}giel in his PhD dissertation \cite{bragiel}. 

In the final section of this note we show how to construct a faithful conditional expectation from $C(SU_q(3)/\T^2)$ onto its subalgebra 
$C(\C P_q^2)$, using integration over the quantum group $U_q(2)$ (realised as a quantum subgroup of $SU_q(3)$). More detailed, algebraic description of the 
noncommutative sphere bundle 
$$ \C P_q^1 \to SU(3)/\T^2 \to \C P_q^2 $$
and its $K$-theory is deferred to the forthcoming paper \cite{BrzSz}.

\section{The quantum flag manifold}

\subsection{The algebra of functions on the quantum $\mathbf{SU(3)}$ group}

For $q\in(0,1)$, the $C^*$-algebra $C(SU_q(3))$ of `continuous
functions' on the quantum $SU(3)$ group is defined by Woronowicz
\cite{w3,w2} as the universal $C^*$-algebra generated by elements
$\{u_{ij}:i,j=1,2,3\}$ such that the matrix $\mathbf{u} = (u_{ij})_{i,j=1}^3$ is unitary and
$$
\sum_{i_1=1}^3\sum_{i_2=1}^3\sum_{i_3=1}^3 E_{i_1i_2i_3}
u_{j_1 i_1}u_{j_2 i_2}u_{j_3 i_3}=E_{j_1 j_2 j_3},
\;\;\;\;\; \forall(j_1,j_2,j_3)\in\{1,2,3\},
$$
where
$$
E_{i_1 i_2 i_3}=\begin{cases} (-q)^{I(i_1,i_2,i_3)} & \text{if}
\; i_r\neq i_s \; \text{for} \; r\neq s, \\ 0 & \text{otherwise,}
\end{cases}
$$
and $I(i_1,i_2,i_3)$ denotes the number of inversed pairs in the
sequence $i_1,i_2,i_3$. As pointed out by Br\k{a}giel \cite{bragiel},
$\{u_{ij}\}$ are coordinate functions of a quantum matrix 
\cite{d,so,frt}. That is, the following relations are also satisfied
\begin{subequations}\label{qmatrix}
\begin{align}
u_{ij}u_{ik} &=  qu_{ik}u_{ij}, \;\;\;\;\; j<k, \label{qmatrix1} \\
u_{ji}u_{ki} &=  qu_{ki}u_{ji}, \;\;\;\;\; j<k, \label{qmatrix2} \\
u_{ij}u_{km} &=  u_{km}u_{ij}, \;\;\;\;\; i<k, \; j>m, \label{qmatrix3} \\
u_{ij}u_{km}-u_{km}u_{ij} &=  (q-q^{-1})u_{im}u_{kj},
\;\;\;\;\; i<k, \; j<m, \label{qmatrix4}
\end{align}
\end{subequations}
with $i,j,k,m\in\{1,2,3\}$. The comultiplication 
$$
\Delta:C(SU_q(3))
\lra C(SU_q(3))\otimes C(SU_q(3))
$$
 is a unital $C^*$-algebra homomorphism
such that
$$ \Delta(u_{ij})=\sum_{k=1}^n u_{ik}\otimes u_{kj}. $$

We denote by $\pol(SU_q(3))$ the $*$-subalgebra of $C(SU_q(3))$ generated
by the $u_{ij}$, $i,j=1,2,3$. Thus $\pol(SU_q(3))$,
the polynomial algebra of $SU_q(3)$, is a dense $*$-subalgebra of $C(SU_q(3))$.

In \cite{bragiel}, Br\k{a}giel described explicitly all irreducible
representations of the algebra $C(SU_q(3))$. There are six families of
these representations, each indexed by elements $(\phi,\psi)$
of the $2$-torus. We denote them by $\pi_0^{\phi,\psi}$, $\pi_{11}^{\phi,\psi}$, 
$\pi_{12}^{\phi,\psi}$, $\pi_{21}^{\phi,\psi}$, $\pi_{22}^{\phi,\psi}$ and
$\pi_3^{\phi,\psi}$. Each of the representations $\pi_*^{\phi,\psi}$ acts on the
Hilbert space $\H_*$, where 
$$\H_0=\C, \;  \H_{11}=\H_{12}=\ell^2(\N), \; 
\H_{21}=\H_{22}=\ell^2(\N^2) \; \mbox{and} \; \H_3=\ell^2(\N^3).
$$
 Each of the
$\pi_*^{\phi,\psi}$ contains  compact operators of $\H_*$
in its image \cite{bragiel}, and thus $C(SU_q(3))$ is a type $I$ algebra.
The kernels of these irreducible representations are primitive
ideals of $C(SU_q(3))$ with the following generators:
\begin{subequations}
\begin{align}
\ker(\pi_3^{\phi,\psi}) & =  \langle \overline{\phi}u_{31}-|u_{31}|,\;
    \overline{\psi}u_{13}-|u_{13}| \rangle, \label{ker1} \\
\ker(\pi_{21}^{\phi,\psi}) & =  \langle u_{31},\;\overline{\phi}u_{21}-|u_{21}|,\;
    \overline{\psi}u_{13}-|u_{13}| \rangle, \label{ker2} \\
\ker(\pi_{22}^{\phi,\psi}) & =  \langle u_{13},\; \overline{\phi}u_{31}-|u_{31}|,\;
    \overline{\psi}u_{12}-|u_{12}| \rangle, \label{ker3} \\
\ker(\pi_{11}^{\phi,\psi}) & =  \langle u_{13},\; u_{31},\; u_{23},\; \overline{\phi}u_{12}
     -|u_{12}|,\; \overline{\psi}u_{21}-|u_{21}| \rangle, \label{ker4} \\
\ker(\pi_{12}^{\phi,\psi}) & =  \langle u_{13},\; u_{31},\; u_{12},\;
     \phi\psi u_{32}-|u_{32}|,\; \overline{\psi}u_{23}-|u_{23}| \rangle, \label{ker5} \\
\ker(\pi_0^{\phi,\psi}) & =  \langle u_{13},\; u_{31},\; u_{12},\; u_{23},\;
     \overline{\phi}u_{11}-1,\; \overline{\psi}u_{22}-1 \rangle. \label{ker6}
\end{align}
\end{subequations}

\subsection{The gauge action and its fixed point algebra}

The family of $1$-dimensional irreducible representations $\pi_0^{\phi,\psi}$
of $C(SU_q(3))$ produces a surjective morphism of compact quantum groups
$$
\hat{\pi}_0:C(SU_q(3))\lra C(\T^2)
$$ 
(the diagonal imbedding of $\T^2$ into $SU_q(3)$),
which gives rise to a gauge coaction of coordinate algebras 
$$
\hat{\mu}:\pol(SU_q(3))\ra \pol(SU_q(3))\otimes
\pol(\T^2), \qquad \hat{\mu}=(\id\otimes\hat{\pi}_0)\circ \Delta_{SU_q(3)}.
$$
Explicitly, on the polynomial algebra $\pol(SU_q(3))$, $\hat{\pi}_0$ is a Hopf $*$-algebra epimorphism,
$$
\hat{\pi}_0:\pol(SU_q(3))\lra \pol(\T^2), \qquad \mathbf{u}\mapsto \begin{pmatrix}U_1 & 0 & 0\cr 0 & U_2 & 0\cr 0 & 0 & U_1^*U_2^*\end{pmatrix},
$$
where $U_1, U_2$ are unitary, group-like generators of the Hopf algebra $\pol(\T^2)$ of polynomials on $\T^2$ (the algebra of Laurent polynomials in two indeterminates). Hence the coaction comes out as
$$
\hat{\mu}:\pol(SU_q(3))\ra \pol(SU_q(3))\otimes
\pol(\T^2), \quad u_{ij}\mapsto \begin{cases} u_{ij}\tens U_j & \text{if} \; j=1,2, \\
u_{ij} \tens (U_1U_2)^{-1}  & \text{if} \; j=3. \end{cases}
$$

Equivalently,
$\mu:\T^2\lra\aut(C(SU_q(3)))$  is given by
$$
z \longmapsto \mu_z, \qquad
  \mu_z(u_{ij}) =\begin{cases} z_j u_{ij} & \text{if} \; j=1,2, \\
(z_1z_2)^{-1}u_{ij} & \text{if} \; j=3. \end{cases}
$$
Here $z=(z_1,z_2)\in\T^2$ and each $z_i$ is a complex number of modulus $1$.
Let $C(SU_q(3)/\T^2)$ be the fixed point algebra of this gauge action, and let
$\pol(SU_q(3)/\T^2)=\pol(SU_q(3))\cap C(SU_q(3)/\T^2)$ be its polynomial
$*$-subalgebra, i.e.\ the subalgebra of coinvariants of $\hat{\mu}$,
$$
\pol(SU_q(3)/\T^2)= \pol(SU_q(3))^{\mathrm{co}\pol(\T^2)} = \{f\in  \pol(SU_q(3))\; |\; \hat{\mu}(f) = f\tens 1\}.
$$

Integration with respect to the Haar measure over $\T^2$ gives rise to
a faithful conditional expectation $\Phi:C(SU_q(3))\ra C(SU_q(3)/\T^2)$, namely
$$
\Phi(x)=\int_{z\in\T^2}\mu_z(x)dz.
$$
If $w$ is a monomial in $\{u_{ij}\}$ then $\Phi(w)$ is either $0$ or $w$.
Thus we have $\Phi(\pol(SU_q(3)))=\pol(SU_q(3)/\T^2)$, and whence
$\pol(SU_q(3)/\T^2)$ is a dense $*$-subalgebra of $C(SU_q(3)/\T^2)$.

There is a third equivalent way of understanding the gauge action, 
which is particularly  useful in determining the freeness of the action (alas we will not employ this point of view in this note): $\pol(SU_q(3))$ is a $\Z^2$-graded algebra with the degrees of the generators given by
$$
\deg(u_{i1}) = (1,0), \; \deg(u_{i2}) = (0,1), \; \deg(u_{i3}) = (-1,-1), \quad i=1,2,3.
$$
From this point of view, $\pol(SU_q(3)/\T^2)$ is the  $(0,0)$-degree part of $\pol(SU_q(3))$.

In what follows, we denote
\begin{equation}\label{gen.flag}
w_{ijk}=u_{i1}u_{j2}u_{k3}, \qquad i,j,k=1,2,3. 
\end{equation}
Clearly, elements $w_{ijk}$ are contained in the polynomial algebra $\pol(SU_q(3)/\T^2)$. 

Let $\rho_*^{\phi,\psi}$
be the restriction to $C(SU_q(3)/\T^2)$ of the representation
$\pi_*^{\phi,\psi}$ of $C(SU_q(3))$.

\begin{lemm}\label{rhoequivalent}
For each $(\phi,\psi)\in\T^2$, the representation $\rho_*^{\phi,\psi}$
is unitarily equivalent to $\rho_*^{1,1}$.
\end{lemm}
\begin{proof}
It follows immediately from formulae (\ref{ker1})--(\ref{ker6}) that
the gauge action $\mu$ on the primitive ideal space is transitive on
each of the six families. Since $C(SU_q(3))$ is of type $I$, irreducible
representations with identical kernels are unitarily equivalent.
Thus, for each $(\phi,\psi)$ there exist $(z_1,z_2)$ such that
$\pi_*^{\phi,\psi}$ is unitarily equivalent to $\pi_*^{1,1}\circ
\mu_{z_1,z_2}$. But $\rho_*^{1,1}\circ\mu_{z_1,z_2}=
\rho_*^{1,1}$, and whence $\rho_*^{\phi,\psi}$ is unitarily
equivalent to $\rho_*^{1,1}$.
\end{proof}

In what follows we use the simplified notation $\rho_*=\rho_*^{1,1}$.

\begin{lemm}\label{rhoirr}
The image of $\rho_*$ contains all the compact operators $\K(\H_*)$
on its space $\H_*$, and thus each $\rho_*$ is irreducible.
\end{lemm}
\begin{proof}
Representation $\rho_0$ is $1$-dimensional and there is nothing to prove in this case.

Considering $\rho_{12}$, given by formulae (14) of \cite{bragiel}, we have
$$
\rho_{12}(w_{132})|N\rangle=-q^{2N+1}|N\rangle.
$$ 
Thus
the image of $\rho_{12}$ contains one-dimensional projections
corresponding to the basis $\{|N\rangle:N\in\N\}$ of $\H_{12}$.
Since 
$$
\rho_{12}(w_{133})|N\rangle=\text{scalar}|N+1\rangle
$$
(in the course of the proof of this lemma we denote by `$\text{scalar}$'
a non-zero constant which may depend on $N,M,L$), it follows that
the image of $\rho_{12}$ contains all the compact operators on
$\H_{12}$. In the case of $\rho_{11}$ the same argument works, since
$\rho_{11}(w_{213})=\rho_{12}(w_{132})$ and
$\rho_{11}(w_{223})|N\rangle=\text{scalar}|N+1\rangle$.

By formulae (12) of \cite{bragiel}, 
$$
\rho_{22}(w_{312})
|N,M\rangle=q^{2(N+M+1)}|N,M\rangle
$$ 
and 
$$
\rho_{22}(w_{132})
|N,M\rangle=-q^{2M+1}(1-q^{2(N+1)})|N,M\rangle.
$$
 It follows that
the image of $\rho_{22}$ contains all one-dimensional projections
corresponding to the basis $\{|N,M\rangle:N,M\in\N\}$ of $\H_{22}$.
We also find that
 $$
 \rho_{22}(w_{112})|N,M\rangle=\text{scalar}
|N-1,M\rangle
$$ 
and 
$$
\rho_{22}(w_{212})|N,M\rangle=\text{scalar}
|N,M-1\rangle, 
$$
and it follows that the image of $\rho_{22}$ contains all
the compact operators on $\H_{22}$. The argument for $\rho_{21}$ is
similar and based on the identities: 
$$
\rho_{21}(w_{231}) =
\rho_{22}(w_{312}),  \qquad \rho_{21}(w_{132})=
\rho_{22}(w_{132}).
$$
$$
 \rho_{21}(w_{131})|N,M
\rangle=\text{scalar}|N-1,M\rangle
$$ 
and
 $$
\rho_{21}(w_{211})
|N,M\rangle=\text{scalar}|N,M-1\rangle.
$$

Finally, considering $\rho_3$, given by formulae (10) of \cite{bragiel}, we have
\begin{subequations}
\begin{align}
\rho_3(|w_{311}|^2)|N,M,L\rangle & =  q^{2(3N+M+L+3)}
(1-q^{2M})|N,M,L\rangle, \label{rho3} \\
\rho_3(w_{111})|N,M,L\rangle & =  \text{scalar}
|N-1,M-1,L\rangle, \label{rho4} \\
\rho_3(w_{211})|N,M,L\rangle & =  \text{scalar}
|N,M-1,L-1\rangle, \label{rho5} \\
\rho_3(w_{311})|N,M,L\rangle & =  \text{scalar}
|N,M-1,L\rangle. \label{rho6}
\end{align}
\end{subequations}
By (\ref{rho3}), the operator $\rho_3(|w_{311}|^2)$ is compact
and its spectral subspace corresponding to the maximal eigenvalue
is spanned by vectors $|0,M,0\rangle$ for which $M$ is a
positive integer such that $q^{2M}(1-q^{2M})$ is maximal.
This space is either one or two-dimensional. In the former case,
the image of $\rho_3$ contains the one-dimensional projection
onto $|0,M_0,0\rangle$, and  formulae (\ref{rho4})--(\ref{rho6})
imply that it contains all the compact operators on $\H_3$.
In the latter case, the image of $\rho_3$ contains the two-dimensional
projection $Q$ onto the span of $|0,M_0,0\rangle$ and $|0,M_0+1,0\rangle$.
Then $Q\rho_3(w_{311})Q$ is a rank one operator, and just as above it
follows from formulae (\ref{rho4})--(\ref{rho6}) that the image of
$\rho_3$ contains all the compact operators on $\H_3$.
\end{proof}

We define $J_*=\rho_*^{-1}(\K(\H_*))$, closed ideals of $C(SU_q(3)/\T^2)$.

\begin{lemm}\label{allrho}
The following properties hold:
\begin{rlist}
\item Representation $\rho_3$ of $C(SU_q(3)/\T^2)$ is faithful.
\item $J_3=\ker(\rho_{21})\cap\ker(\rho_{22})$.
\item $J_{21}=J_{22}=\ker(\rho_{21})+\ker(\rho_{22})
   =\ker(\rho_{11})\cap\ker(\rho_{12})$.
\item $J_{11}=J_{12}=\ker(\rho_{11})+\ker(\rho_{12})
   =\ker(\rho_0)$.
\end{rlist}
\end{lemm}
\begin{proof}
We have
$$ \ker(\rho_*)=C(SU_q(3)/\T^2)\cap\bigcap_{\phi,\psi}\ker(\pi_*^{\phi,\psi}). $$
Using formulae (\ref{ker1})--(\ref{ker6}) we see that $\cap_{\phi,\psi}
\ker(\pi_*^{\phi,\psi})$ are ideals of $C(SU_q(3))$ with the following
sets of generators:
\begin{subequations}
\begin{align}
\bigcap_{\phi,\psi}\ker(\pi_3^{\phi,\psi}) & = 
   \langle 0 \rangle, \label{kerphi1} \\
\bigcap_{\phi,\psi}\ker(\pi_{21}^{\phi,\psi}) & =  \langle u_{31} \rangle,
   \;\;\;\;\;\;\; \bigcap_{\phi,\psi}\ker(\pi_{22}^{\phi,\psi})=
   \langle u_{13} \rangle, \label{kerphi2} \\
\bigcap_{\phi,\psi}\ker(\pi_{11}^{\phi,\psi}) & = 
   \langle u_{13},\; u_{31},\; u_{23} \rangle,\label{kerphi3}\\
   \bigcap_{\phi,\psi}\ker(\pi_{12}^{\phi,\psi}) & = 
   \langle u_{13},\; u_{31},\; u_{12} \rangle, \label{kerphi4}
\end{align}
\end{subequations}
On the other hand, $J_*=\ker(\rho_*)\cap\langle x_* \rangle$, where
$\langle x_* \rangle$ is the ideal of $C(SU_q(3))$ generated by
$x_*$ such that $\pi_*^{\phi,\psi}(x_*)$ is a non-zero element of
$\K(\H_*)$ for all $\phi,\psi$. For example, we can take $x_3=
u_{31}u_{13}$, $x_{21}=u_{13}$, $x_{22}=u_{31}$, $x_{11}=u_{12}$
and $x_{12}=u_{23}$.

Now (i) follows from formula (\ref{kerphi1}). Claim (ii) follows
from (\ref{kerphi2}) and the identity $\langle u_{31} \rangle \cap
\langle u_{13} \rangle = \langle u_{31}u_{13} \rangle$. The latter follows
from the fact that both $u_{13}$ and $u_{31}$ either commute or $q$-commute
with every generator of $C(SU_q(3))$.

The identity $J_{21}=J_{22}=\ker(\rho_{21})+\ker(\rho_{22})$ follows
from (\ref{kerphi2}). For the remaining part of claim (iii), it suffices
to show that 
$$ 
\langle u_{13},\; u_{31} \rangle = \langle u_{13},\; u_{31},\;
u_{12} \rangle \cap \langle u_{13},\; u_{31},\; u_{23} \rangle .
$$
 To this end,
we first note that modulo the ideal $ \langle u_{13},\; u_{31} \rangle $ both
$u_{12}$ and $u_{23}$ either commute or $q$-commute with every generator
of $C(SU_q(3))$. Thus 
$$
\langle u_{13},\; u_{31},\; u_{12} \rangle \cap
\langle u_{13},\; u_{31},\; u_{23} \rangle = \langle u_{13},\; u_{31},\;
u_{12}u_{23} \rangle,
$$
and it suffices to verify that $u_{12}u_{23}\in
\langle u_{13},\; u_{31} \rangle $. By formulae (11)  of
\cite{bragiel}, we have $\pi_{21}^{\phi,\psi}(u_{12}u_{23})\in\K(\H_{21})$ for
all $\phi,\psi$, and the claim follows.

The identity $J_{11}=J_{12}=\ker(\rho_{11})+\ker(\rho_{12})$ follows from
(\ref{kerphi3}) and \eqref{kerphi4}. For the remaining part of claim (iv), we must prove that
$$
\ker(\rho_0)=C(SU_q(3)/\T^2)\cap\langle u_{13},\; u_{31},\; u_{23},\;
u_{12} \rangle.
$$
However,  as shown in \cite{bragiel}, $\oplus_{\phi,\psi}
\pi_0^{\phi,\psi}$ is faithful on the quotient of $C(SU_q(3)$ by
$\langle u_{13},\; u_{31},\; u_{23},\; u_{12} \rangle$, and the claim follows.
\end{proof}

In the following corollary we summarise properties of the algebra of continuous functions on the quantum flag manifold $SU_q(3)/\T^2$.

\begin{coro}\label{af}
The $C^*$-algebra $C(SU_q(3)/\T^2)$ has the following properties.
\begin{enumerate}
\item It has a composition series with factors: $\K$, $\K\oplus\K$,
$\K\oplus\K$, $\C$.
\item It is $AF$ and of type $I$.
\item Its $K$-groups are $K_0\cong\Z^6$ and $K_1=0$.
\item $\{\rho_*\}$ is a complete set of representatives
(up to unitary equivalence) of its irreducible representations.
\item Each irreducible representation of $C(SU_q(3)/\T^2)$ extends
to an irreducible representation of $C(SU_q(3))$ acting on the
same Hilbert space. \item Its primitive ideal space consists of
six elements $\{\ker(\rho_*)\}$, with topology determined by the
following closure operation.
\begin{enumerate}
\item The point $\ker(\rho_3)$ is dense in the entire space.
\item The closures of $\ker(\rho_{21})$ and $\ker(\rho_{22})$,
respectively, consist of the union of itself and
$\{\ker(\rho_0),\ker(\rho_{11}),\ker(\rho_{12})\}$. 
\item The closures of $\ker(\rho_{11})$ and $\ker(\rho_{12})$,
respectively, consist of the union of itself and $\ker(\rho_0)$.
\item The point $\ker(\rho_0)$ is closed.
\end{enumerate}
\end{enumerate}
\end{coro}

\section{Towards a noncommutative sphere bundle}

The classical  flag manifold $SU(3)/\T^2$ has the structure of a fibre bundle
with the base space $\C P^2$ and the fibre $\C P^1\cong S^2$. Therefore, it
is natural to expect that the quantum  flag manifold $SU_q(3)/\T^2$ should
have an analogous structure of a noncommutative `fibre bundle'
\begin{equation}\label{bundle}
\C P_q^1\lra SU_q(3)/\T^2 \lra \C P_q^2.
\end{equation}

It is not entirely clear how to reinterpret the ``bundle'' from (\ref{bundle}) in the noncommutative setting. However, as a minimum, we should have 
a projection (conditional expectation) from the algebra of ``functions on the total space'' $C(SU_q(3)/\T^2)$ onto the algebra of ``functions on the 
base space'' $C(\C P_q^2)$. So we begin by constructing such a conditional expectation.

The algebra $C(\C P_q^2)$ is a $C^*$-subalgebra of $C(SU_q(3)/\T^2)$ in a natural
way as follows (cf. \cite{vs}). The $C^*$-subalgebra of
$C(SU_q(3))$ generated by the first column matrix elements of $\mathbf{u}$, i.e.\ 
$u_{11}$, $u_{21}$ and $u_{31}$, may be identified with the
$C^*$-algebra $C(S^5_q)$ of continuous functions on the quantum
$5$-sphere. This $C^*$-subalgebra is invariant under the gauge
action $\mu$ of $\T^2$ on $C(SU_q(3))$. When restricted to
 $C(S^5_q)$, $\mu$ reduces to the generator-rescaling
circle action $u_{j1}\mapsto zu_{j1}$, $z\in\T$, whose fixed
point algebra is $C(\C P_q^2)$ (cf. \cite{vs,hs}). Thus,
in the setting of the present article, we have
$$
C(\C P_q^2)=C(SU_q(3)/\T^2)\cap C^*(u_{11},u_{21},u_{31}).
$$

In order to construct the desired conditional expectation 
$$
E: C(SU_q(3)/\T^2) \to C(\C P_q^2), 
$$
we will use integration over a quantum subgroup of $SU_q(3)$ isomorphic to the quantum unitary group $U_q(2)$. Indeed, 
recall from \cite{frt} or \cite{h} that $U_q(2)$ is a compact matrix quantum group with the $C^*$-algebra of continuous functions $C(U_q(2))$ generated densely by three elements $u,\alpha,\gamma$, organised  into a fundamental unitary matrix 
$$
\mathbf{v} =  \begin{pmatrix} u & 0 & 0 \cr
0 & \alpha & -q\gamma^*u^*\cr 
0 & \gamma & \alpha^*u^*
\end{pmatrix}.
$$
The generator $u$ is central, while
$$
\alpha\gamma = q \gamma\alpha,  \qquad \gamma\gamma^* = \gamma^*\gamma .
$$
The unitarity of $\mathbf{v}$ implies that $u$ is unitary, while $\alpha$ and $\gamma$ satisfy the remaining $SU_q(2)$ (cf. \cite{w1}) 
$q$-commutation rules
$$
\alpha\gamma^* = q \gamma^*\alpha, \qquad \alpha^*\alpha + \gamma\gamma^* = 1, \qquad \alpha\alpha^* + q^2 \gamma\gamma^*=1.
$$

As shown in \cite{BrzSz}, the $*$-homomorphism 
$$
\pi : \pol (SU_q(3)) \lra \pol (U_q(2)), \qquad \mathbf{u} \mapsto  \mathbf{v},
$$
is an epimorphism of compact quantum groups, and thus we  obtain a right coaction
\begin{equation}\label{coact.u2.su3}
\begin{aligned}
\varrho_{SU_q(3)} &: C(SU_q(3))\lra C(SU_q(3))\otimes C(U_q(2)), \\ \varrho_{SU_q(3)} &= (\id\otimes \pi)\circ \Delta_{SU_q(3)}. 
\end{aligned}
\end{equation}

One immediately checks that 
$$
\varrho_{SU_q(3)} \circ \mu_z = (\mu_z\otimes\id)\circ \varrho_{SU_q(3)}, 
$$
for all $z\in\T^2$, and this implies that the restriction of $\varrho_{SU_q(3)}$ to $C(SU_q(3)/\T^2)$ yields the  coaction
$$
\varrho_{SU_q(3)/\T^2} : C(SU_q(3)/\T^2)\lra C(SU_q(3)/\T^2)\otimes C(U_q(2)).
$$
Consequently,  
$$
(\id\otimes\h)\circ\varrho_{SU_q(3)/\T^2} : C(SU_q(3)/\T^2) \to C(SU_q(3)/\T^2)^{co\, U_q(2)}
$$
is a faithful conditional expectation. Here $\h$ denotes the Haar state on $C(U_q(2))$ and 
$$ 
C(SU_q(3)/\T^2)^{co\, U_q(2)} = \{ a\in  C(SU_q(3)/\T^2) \mid \varrho_{SU_q(3)/\T^2}(a) = a\otimes 1\} 
$$ 
is the $C^*$-subalgebra of coinvariants. It is shown in \cite{BrzSz} that 
$$
C(SU_q(3)/\T^2)^{co\, U_q(2)} = C(SU_q(3)) \cap C^*(u_{11}, u_{21}, u_{31}), 
$$
and thus 
\begin{equation}\label{expect}
E =   (\id\otimes\h)\circ \varrho_{SU_q(3)/\T^2} 
\end{equation} 
is the desired faithful conditional expectation from the algebra $C(SU_q(3)/\T^2)$ onto $C(\C P_q^2)$. 

In order to compute the conditional expectation value \eqref{expect} it is useful or, indeed necessary, to have an explicit description of the Haar state on $C(U_q(2))$. In fact, it is sufficient to have such a description on the dense subalgebra $\pol (U_q(2))$ of $C(U_q(2))$. One way of obtaining the Haar measure is first to realise that $\pol (U_q(2))$ is a right $\pol(SU_q(2))$-comodule algebra (i.e.\ the quantum group $SU_q(2)$ acts on $U_q(2)$) with fixed points equal to $\pol (U(1))$ and then to compose the Haar integrals on $\pol(SU_q(2))$ and $\pol(U(1))$ (both well-known, the first one  described by Woronowicz in \cite{w2}). 

The coaction $\varrho_{U_q(2)}$ of $\pol(SU_q(2))$ on $\pol (U_q(2))$ is induced from the Hopf-algebra projection
$$
\xymatrix{  \pol(U_q(2)) \ar@{->>}[r]^p  & \pol(SU_q(2))\, ,\\
 {\begin{pmatrix} u & 0 & 0 \cr
0 & \alpha & -q\gamma^*u^*\cr 
0 & \gamma & \alpha^*u^*
\end{pmatrix}}\ar@{|->}[r] & {\begin{pmatrix} 1 & 0 & 0 \cr
0 & \alpha & -q\gamma^*\cr 
0 & \gamma & \alpha^*
\end{pmatrix}}, }
$$
by 
\begin{equation}\label{su2-coact}
\varrho_{U_q(2)} = (\id\otimes p)\circ \Delta_{U_q(2)}: \pol(U_q(2))\lra \pol(U_q(2))\ot \pol(SU_q(2)).
\end{equation}
 As $*$-algebra the coinvariants of this coaction are generated by the unitary $u$, and hence are isomorphic to $\pol (U(1))$. The Haar functional $\h$ on $\pol(U_q(2))$ (and, consequently, on $C(U_q(2))$) is the composite
$$
\h: \xymatrix{\pol(U_q(2))\ar[r]^-{\varrho_{U_q(2)}} & \pol(U_q(2))\ot \pol(SU_q(2))\ar[rr]^-{\id \ot \h_{SU_q(2)}} &&  \pol (U(1))\ar[r]^-{\h_{U(1)}} & \C.}
$$
Here $\h_{SU_q(2)}$ is the Haar measure on the quantum group $SU_q(2)$ given on the standard basis of $\pol(U_q(2))$ as
\begin{equation}\label{haar.su2}
\h_{SU_q(2)}\left(\alpha^k\gamma^m\gamma^{*n}\right) = \delta_{k0}\,\delta_{mn}\, \frac{q^2 -1}{q^{2n+2} - 1}, \quad \mbox{for all $k\in \Z$, $m,n\in \N$},
\end{equation}
where we use the convention that, for $k<0$,  $x^k= (x^{*})^{-k}$; see \cite[Appendix~A1]{w2}. The Haar functional $\h_{U(1)}$ on the standard basis of $\pol (U(1))$ is given by
\begin{equation}\label{haar.u1}
\h_{U(1)} \left(u^k\right) = \delta_{k0}, \quad \mbox{for all $k\in \Z$}.
\end{equation}
Combining formulae \eqref{su2-coact}--\eqref{haar.u1} we thus obtain an explicit expression for the Haar functional on $\pol(U_q(2))$,
\begin{equation}\label{haar.u2}
\h\left(\alpha^ku^l\gamma^m\gamma^{*n}\right) = \delta_{k0}\,\delta_{l0}\,\delta_{mn}\, \frac{q^2 -1}{q^{2n+2} - 1}, \quad \mbox{for all $k,l\in \Z$, $m,n\in \N$}.
\end{equation}

With the explicit formula \eqref{haar.u2} at hand we can now compute the value of the conditional expectation \eqref{expect} on the elements  \eqref{gen.flag} of the quantum flag variety algebra $\pol(SU_q(3)/\T^2)$ densely included in the $C^*$-algebra of continuous functions $C(SU_q(3)/\T^2)$. In view of the fact that the coaction $\varrho_{SU_q(3)/\T^2}$ is the restriction of the map \eqref{coact.u2.su3} one easily finds that
\begin{multline*}
\varrho_{SU_q(3)} \left(w_{ijk}\right) = w_{ijk} \ot 1 + u_{i1}u_{j2}u_{k2} \ot uv_{22}v_{23} \\
 + u_{i1} \left(u_{j3}u_{k2} + q u_{j2}u_{k3}\right)\ot u v_{32} v_{23}
 + u_{i1}u_{j3}u_{k3} \ot u v_{32} v_{33}\\
  = w_{ijk} \ot 1 -q u_{i1}u_{j2}u_{k2} \ot \alpha \gamma^* \\
  - qu_{i1} \left(u_{j3}u_{k2} + q u_{j2}u_{k3}\right)\ot \gamma\gamma^*
 + u_{i1}u_{j3}u_{k3} \ot \gamma\alpha^*.
\end{multline*}
Now, the application of $\id\ot \h$ together with the commutation rules \eqref{qmatrix} yield,
$$
E(w_{ijk}) = \frac{w_{ijk} - w_{ikj}}{1+q^2}.
$$

\section{Conclusions}
In this short note we have studied representations and the structure of the algebra of continuous functions on the quantum flag manifold $SU_q(3)/\T^2$ obtained as the fixed points of the gauge action of the classical two-torus on the quantum $SU(3)$-group. We have also indicated that the  quantum flag manifold $SU_q(3)/\T^2$ can be interpreted as the total space of a quantum sphere bundle over the quantum projective space $\C P_q^2$, and we have presented an explicit formula for a faithful conditional expectation from $C(SU_q(3)/\T^2)$ onto $C(\C P_q^2)$. The detailed analysis of this bundle is presented in \cite{BrzSz}.


\begin{thebibliography}{88}

\bibitem{bragiel} K. Br\k{a}giel, 
{\em The twisted $SU(3)$ group. Irreducible $*$-representations of the $C^*$-algebra $C(S_\mu U(3))$}, 
Lett. Math. Phys. {\bf 17} (1989), 37--44.




\bibitem{BrzHaj:Che}
T.~Brzezi{\'n}ski and P.~M.~Hajac,
{\em The Chern-Galois character,}
C.\ R.\ Math.\ Acad.\ Sci.\ Paris  {\bf 338}  (2004),  113--116

\bibitem{BrzMaj:gau}
T.~Brzezi{\'n}ski and S.~Majid, 
{\em Quantum group gauge theory on quantum spaces,}
Commun.\ Math.\ Phys.\ {\bf 157}   (1993),  591--638. 
Erratum: {\bf 167}  (1995), 235. 

\bibitem{BrzSz} T. Brzezi{\'n}ski and W. Szyma{\'n}ski, 
{\em The quantum flag manifold $SU_q(3)/\T^2$ as an example of a noncommutative sphere bundle}, 
in preparation.

\bibitem{d} V. G. Drinfeld, 
{\em Quantum groups}. 
Proceedings of the International Congress of Mathematicians, Vol. 1, 2 (Berkeley, 1986), 798--820, Amer. Math. Soc., Providence, 1987.

\bibitem{h} P. M. Hajac, 
{\em Strong connections on quantum principal bundles},  
Commun. Math. Phys. {\bf 182}  (1996),  579--617. 

\bibitem{hs} J. H. Hong and W. Szyma\'{n}ski, 
{\em Quantum spheres and projective spaces as graph algebras}, 
Commun. Math. Phys. {\bf 232} (2002), 157--188.

\bibitem{ks} A. Klimyk and K. Schm\"{u}dgen, 
Quantum groups and their representations, Springer, Berlin, 1997.

\bibitem{NevTus:hom} S. Neshveyev and L. Tuset, {\em Quantized algebras of functions on homogeneous spaces with Poisson stabilizers}, Comm. Math. Phys. {\bf 312} (2012),  223--250.

\bibitem{frt} N. Reshetikhin, L. Takhtajan and L. Faddeev, 
{\em Quantization of Lie groups and Lie algebras}, 
Leningrad Math. J. {\bf 1} (1990), 193--226.

\bibitem{so} Y. S. Soibelman, 
{\em Irreducible representations of the algebra of functions on the quantum group $SU(n)$ and Schubert cells}, 
Dokl. Akad. Nauk SSSR {\bf 307} (1989), 41--45.

\bibitem{StoDij:fla} J. V. Stokman and M. S.  Dijkhuizen, {\em Quantized flag manifolds and irreducible $*$-representations}, Comm. Math. Phys. {\bf 203} (1999),  297--324. 

\bibitem{vs} L. L. Vaksman and Y. S. Soibelman, 
{\em Algebra of functions on quantum $SU(n+1)$ group and odd dimensional quantum spheres}, 
Algebra i Analiz {\bf 2} (1990), 101--120.

\bibitem{w1} S. L. Woronowicz, 
{\em Twisted $SU(2)$ group. An example of a non-commutative differential calculus}, 
Publ. Res. Inst. Math. Sci. {\bf 23} (1987), 117--181.

\bibitem{w2} S. L. Woronowicz, 
{\em Compact matrix pseudogroups}, 
Commun. Math. Phys. {\bf 111} (1987), 613--665.

\bibitem{w3} S. L. Woronowicz, 
{\em Tannaka-Krein duality for compact matrix pseudogroups. Twisted $SU(N)$ groups}, 
Invent. Math. {\bf 93} (1988), 35--76.

\end{thebibliography}
\end{document}